\documentclass[12pt]{article}
\usepackage{a4, amsmath,amssymb,amsbsy,amsfonts,amsthm,latexsym,
             amsopn,amstext,amsxtra,euscript,amscd}

\newtheorem{lem}{Lemma}
\newtheorem{lemma}[lem]{Lemma}

\newtheorem{thm}{Theorem}
\newtheorem{theorem}[thm]{Theorem}

\newtheorem{ques}{Question}

\def\\{\cr}
\def\({\left(}
\def\){\right)}
\def\[{\left[}
\def\]{\right]}
\def\<{\langle}
\def\>{\rangle}
\def\fl#1{\left\lfloor#1\right\rfloor}
\def\rf#1{\left\lceil#1\right\rceil}

\def\cG{\mathcal G}

\def\F{\mathbb{F}}
\def\Fp{\mathbb{F}_p}

\def\cG{{\mathcal G}}
\def\cH{{\mathcal H}}

\def\cQ{{\mathcal Q}}

\def\cU{{\mathcal U}}
\def\cV{{\mathcal V}}

\def\ord{\mathrm{ord\,}}

\def\mand{\qquad \mbox{and} \qquad}

\begin{document}

\title{\bf On Gaps Between Primitive Roots in the
Hamming Metric}

\author{
{\sc Rainer Dietmann}\\
{Department of Mathematics}\\
{Royal Holloway,
University of London} \\ 
{Egham, Surrey, TW20 0EX, United Kingdom} \\
{\tt Rainer.Dietmann@rhul.ac.uk}
\\
\and
{\sc Christian Elsholtz} \\
{Institute of Analysis and Computational Number Theory} \\
{Technische  Universit\"at  Graz} \\ 
{Steyrergasse 30}\\
{Graz, A-8010, Austria } \\
{\tt elsholtz@math.tugraz.at} \\
\and
{\sc Igor E.~Shparlinski} \\
{Department of Computing, Macquarie University} \\
{Sydney, NSW 2109, Australia} \\
{\tt igor.shparlinski@mq.edu.au} 
}

\pagenumbering{arabic}
\maketitle

\begin{abstract} We consider a modification 
of the classical number theoretic question about the gaps
between consecutive primitive roots modulo a prime $p$,
which by the well-known
result of Burgess are known to be at most $p^{1/4+o(1)}$.
Here we measure the distance in the Hamming metric and show that
if $p$ is a sufficiently large  $r$-bit prime, 
then for 
 any integer $n \in [1,p]$ one can obtain a primitive root modulo $p$
by changing  at most $0.11002786\ldots r$ binary digits of $n$. 
This  is stronger than what can be deduced from the Burgess result. 
Experimentally, the number of necessary bit changes is very small.
We also show that each Hilbert cube contained in the complement
of the primitive roots modulo $p$ has dimension at most $O(p^{1/5+\epsilon})$,
improving on previous results of this kind.
\end{abstract}

\section{Introduction}

Let $p$ be a fixed prime number.  Studying the gaps 
between consecutive quadratic non-residues and primitive roots modulo 
$p$ is a classical number theoretic question, where however 
not much progress has been made since the work of Burgess~\cite{Burg1,Burg2}
that implies that these gaps are at most $p^{1/4+o(1)}$.

Here we consider a modification of this question where the 
distances are measured in the \textit{Hamming metric}
(see for example ~\cite[Section~1.1]{MS}). More specifically,
we denote by $\Delta_p$ the smallest number $s$ such that  
for any integer $n \in [1,p]$ one can change at most $s$ binary 
digits of $n$ in order to get a primitive root modulo $p$. 

We use the ideas of~\cite{OstShp}, which in turn expand on 
those of~\cite{BaCoSh}, to estimate 
character sums over integers that are close in the Hamming 
metric to a given integer $n$  and then derive an 
estimate on $\Delta_p$. As a primitive root is also quadratic non-residue, 
$\Delta_p$ also gives a bound 
on the gaps between quadratic non-residues in the Hamming metric.

Let $\rho_0 = 0.11002786\ldots$ be the unique root of the equation 
\begin{equation}
\label{eq:rho}
H(\rho) = 1/2, \qquad 0 \le \rho \le 1/2, 
\end{equation}
where
$$
H(\gamma) = \frac{- \gamma \log \gamma  -  (1-\gamma) 
\log (1-\gamma)}{\log 2}, \quad 0 < \gamma < 1,
$$
denotes the {\it binary entropy function\/} (see ~\cite[Section~10.11]{MS}).

\begin{theorem}
\label{thm:Deltap} 
 We have $\Delta_p \le (\rho_0 + o(1)) r$
as $p \to \infty$, where  $r$ is the number of 
 binary digits of $p$. \end{theorem}

Note that an immediate application of the Burgess result~\cite{Burg1} would only
give $0.25$ in place of $\rho_0$. 

It is also interesting to study the sparsest primitive root and quadratic 
non-residue. More precisely, let $W_p$ be the smallest Hamming weight
(see~\cite[Section~1.1]{MS}) 
of the binary expansion of the primitive roots  
$g\in  \{1, \ldots, p-1\}$ modulo $p$.
For $p \ne 2$, we define $w_p$ analogously with respect to quadratic 
non-residues modulo $p$.
Since for $p>2$
a primitive root modulo $p$ is necessarily a quadratic non-residue
modulo $p$, we have
$$
w_p \le W_p \le \Delta_p.
$$
We are not able to improve the above bound for $W_p$, however, using a
recent result of~\cite{BaGaHBSh} we obtain a more precise estimate on $w_p$.

Let 
\begin{equation}
\label{eq:theta}
\vartheta_0 = \frac{1}{8 \sqrt{e}} = 0.07581633\ldots.
\end{equation}

\begin{theorem}
\label{thm:SpareseNonRes} 
 We have
 $w_p \le \(\vartheta_0 + o(1)\) r$,
as $p \to \infty$, where  $r$ is the number of 
 binary digits of $p$. 
 \end{theorem}

Again, a direct application of the Burgess result~\cite{Burg1} 
gives a weaker bound,
namely 
$$
w_p \le \(\frac{1}{4\sqrt{e}} + o(1)\)r.
$$

Finally, we also consider the distribution of primitive roots in 
so-called {\it Hilbert cubes\/}:
 For $a_0, a_1,
\ldots, a_d \in \F_p$ write
\begin{equation}
\label{hilbert_cube}
\cH(a_0; a_1, \ldots, a_d) = \left\{ a_0 + \sum_{i=1}^d \vartheta_i a_i:
\vartheta_i \in \{0,1\}\right\}.
\end{equation}
As in~\cite{HS}, we define $f(p)$ as the largest $d$ such that 
there are $a_0, a_1, \ldots, a_d \in \F_p$ with   pairwise distinct
$a_1, \ldots, a_d$ such that
$\cH(a_0; a_1, \ldots, a_d)$ does not contain a quadratic non-residue 
modulo $p$. Furthermore, we define $F(p)$ 
as the largest $d$ 
 such that there are $a_0, a_1, \ldots, a_d \in \F_p$ with pairwise distinct
$a_1, \ldots, a_d$ such that $\cH(a_0; a_1, \ldots, a_d)$ does not contain a 
primitive root modulo $p$. Also, for the complementary sets we define
$\overline{f}(p)$ and $\overline{F}(p)$
as the largest $d$ such that  $\cH(a_0; a_1, \ldots, a_d)$
is entirely in the set of quadratic non-residues 
respectively primitive roots modulo $p$.

As a primitive root is a non-residue,
and as a set of residues becomes a set of
non-residues by multiplication with one fixed non-residue 
(recall that $0$ is neither residue, nor non-residue) we have:
$$
\overline{F}(p)\leq
\overline{f}(p)= f(p) \le F(p).
$$
 Hegyv\'{a}ri and S\'{a}rk\"{o}zy~\cite[Theorem~2]{HS} give the bound 
 $f(p) < 12 p^{1/4}$. 
Here we improve the exponent and also  extend the result to $F(p)$.

\begin{theorem} 
\label{thm:HilbCube} We have  $F(p) \le p^{1/5+o(1)}$ 
as $p \to \infty$. 
\end{theorem}
As for primes with $(p-1)/2$ also a prime
the set of non-residues is the same as the set of primitive roots one should
not expect that upper bounds on 
$\overline{F}(p)$ are generally better than those for $F(p)$.

From their result, Hegyv\'{a}ri and S\'{a}rk\"{o}zy~\cite{HS} 
give an application to the
maximal dimension $d$ of Hilbert cubes in the set of integer squares.
We do not follow this path here, but remark that  the first two authors 
have recently
improved the bound on $d$ using a different 
method, see~\cite{DietmannandElsholtz}.

It is likely that the bound $p^{1/5+o(1)}$ is far from the truth. 
One may conjecture a bound of $p^{o(1)}$ or even $(\log p)^C$ for some positive
constant $C$. 
Indeed, for the easier problem of \textit{subset sums} where
$a_0=0$, for $p \equiv \pm 3 \pmod 8$ and subset sums
avoiding quadratic non-residues modulo $p$,
Csikv{\'a}ri~\cite[Corollary 2.2]{Csi} has obtained an upper
bound $\log p/\log 2$.
However, it may be difficult to prove a bound of this type for the
general case.
It has been observed in~\cite{HS, Csi} that improving the
bound on $f(p)$ 
to $p^{\alpha}$ with
 $\alpha<\vartheta_0$, where $\vartheta_0$ is
 given by~\eqref{eq:theta}, is impossible without 
improving the Burgess bound~\cite{Burg1} on the smallest quadratic non-residue.
To see this,
one simply constructs a Hilbert cube consisting of many small elements
$a_i=i$, $i=0, \ldots , d$ where $d=f(p)$. Here the elements of the Hilbert cube
are at most 
$$
1+2+ \ldots + d<d^2=f(p)^2<p^{2\alpha}.
$$ 
Hence $ 2\alpha \geq 1/(4e^{-1/2})+o(1)$ 
unless one improves on the Burgess bound. 
In fact, the same argument shows that in the bound on $F(p)$ 
one cannot go beyond $1/8= 0.125$ in the exponent 
without improving the Burgess bound $g(p) \le p^{1/4+o(1)}$ on the least primitive root
$g(p)$ modulo $p$.
In this context, let us recall that, assuming the
Generalised Riemann Hypothesis (GRH),  Shoup~\cite{Shoup}
has  proved a bound of 
\begin{equation}
\label{ShoupBound}
g(p) = O((\log p)^6). 
\end{equation}

Finally, let us remark that Theorem \ref{thm:HilbCube} immediately gives
$\Delta_p \le (0.2+o(1)) r$, which is stronger than the bound 
$\Delta_p \le (0.25+o(1)) r$
resulting from applying the Burgess bound, but weaker than our result
obtained in Theorem~\ref{thm:Deltap} making use of a direct
application of exponential sums.

\section{Preparations}
\label{sec:Preps}

Throughout the paper the implied constants in the symbols ``$O$'', 
``$\ll$'' and ``$\gg$'' may depend on an integer parameter $\nu \ge 1$.  We recall that
the expressions $A \ll B$, $B \gg A$ and $A=O(B)$ are each equivalent to the
statement that $|A|\le cB$ for some constant $c$.  As usual,
$\log z$ denotes the natural logarithm of $z$.

The letter $p$ (possibly subscripted) always denotes a prime.

We also use $\F_p$ to denote the finite field of $p$ elements.

We need the following well-known statement (see, for example,~\cite[Section~10.11, Lemma 7]{MS}):

\begin{lemma}
\label{lem:Binom}
For any integers $k \ge \ell \ge 0$,
$$ 
\binom{k}{\ell } = 2^{k H(\ell /k)+ o(k)}.
$$
\end{lemma}

We note that $\overline \chi(z) =
\chi(z^{p-2})$ for $z \in \F_p^*$ and a multiplicative character $\chi$
of $\F_p^*$. 

We need the following statement which follows immediately from the Weil bound,  
see~\cite[Chaper~11]{IwKow}, and which is essentially~\cite[Theorem~2]{MaSa}.

\begin{lemma}
\label{lem:Weil-Char}
For  any multiplicative character
$\chi$ of $\F_p^*$ of order $m\ge 2$, any integers $M$ and $K$ with $1\le K<p$,
and any polynomial $F(U)\in\Fp[U]$
with  $d$ distinct roots (of arbitrary multiplicity)
such that $F(U)$ is not the $m$-th power of a rational function,
we have
$$\sum_{u=M+1}^{M+K} \chi(F(u))\ll d p^{1/2} \log p.$$
\end{lemma}

The following result  is a combination of  the bounds of 
P{\'o}lya-Vinogradov  (for $\nu =1$) and Burgess
(for $\nu\ge2$),
see~\cite[Theorems~12.5 and 12.6]{IwKow}.

\begin{lemma}
   \label{lem:PVB} For arbitrary  integers
$W$ and $Z$ with  $1 \le Z \le p$, for an
arbitrary non-principal multiplicative character $\chi$ of $\F_p^*$,
and for an arbitrary positive integer $\nu$, we have
$$
\left| \sum_{z = W+1}^{W+Z}
\chi(z)\right|  \le Z^{1 -1/\nu} p^{(\nu+1)/4\nu^2 + o(1)}.
$$
\end{lemma}

As usual, we use $\mu(d)$ and $\varphi(d)$ to denote the
M{\"o}bius and the Euler functions of an integer $d \ge 1$, respectively. 
We now mention the following well-known characterisation of primitive roots
modulo $p$ which follows from the inclusion-exclusion principle and
the orthogonality property of characters 
(see, for example,~\cite[Exercise~5.14]{LN}).

\begin{lemma}
\label{lem:prim root}
For any integer $a$, we have
$$
\frac{\varphi(p-1)}{p-1}\sum_{d \mid p-1} \frac{\mu(d)}{\varphi(d)}
\sum_{\ord \chi = d} \chi(a)
= \left\{\begin{array}{ll}
1,& \text{if $a$ is a primitive root modulo $p$,}\\
0,& \text{otherwise,}
\end{array}
\right.
$$
where the inner sum is taken over all $\varphi(d)$ multiplicative characters $\chi$ modulo $p$ 
of order $d$. 
\end{lemma}

We also recall the following bound~\cite[Theorem 2.1]{BaGaHBSh}
on short sums of the Legendre symbol $(n/p)$ modulo~$p$.

\begin{lemma}
\label{lem:PosPropNonres} For every $\varepsilon>0$ there exists $\delta>0$
such that, for all sufficiently large primes $p$, the bound
$$
\left|\sum_{n\le N}(n/p)\right|\le(1-\delta)N
$$
holds for all integers $N$ in the range $p^{1/(4\sqrt{e})+\varepsilon}\le
N\le p$.
\end{lemma}

Finally, for our result on Hilbert cubes avoiding primitive roots, we
make use of a recent result by Schoen~\cite[Theorem~3.3]{Sch}
in additive combinatorics.
Note that his result is actually only stated
for subset sums rather than Hilbert cubes, but 
this slight generalisation follows immediately.

\begin{lemma}
\label{addcomb}
For any $a_0 \in \F_p$ and pairwise distinct $a_1, \ldots, a_d \in \F_p$
such that $d \ge 8(p/\log p)^{1/D}$, where $D$ is an integer satisfying
$$
0<D\le \sqrt{\frac{\log p}{2 \log \log p}},
$$  
the Hilbert cube~\eqref{hilbert_cube} contains an arithmetic progression
of length $L$ where
$$
  L \ge 2^{-10} (d/\log p)^{1+1/(D-1)}.
$$
\end{lemma}

\section{Double Multiplicative Character Sums}
\label{sec:CharSums}

In the following, for a fixed prime $p$ we write $r$ for the non-negative
positive integer such that $2^r < p \le 2^{r+1}$.
Given integers $n\in [1,p]$, $k \in [1,r]$ and $l \le k$,
we denote by $\cU_{k,\ell}(n)$ the set of positive integers $u < 2^k$
whose binary expansions differ from the  $k$ most significant binary digits 
of $n$ in exactly $\ell$ positions (if necessary,
we append some leading zeros to the
binary expansion of $n$ to guarantee that it is of length $r+1$).
Furthermore, for $m \le r-k$ we denote by 
$\cV_{k,m}(n)$ the set of positive integers $v < 2^{r-k+1}$
whose binary expansions differ from the  $r-k+1$ least binary digits 
of $n$ in exactly $m$ positions.

Obviously,
\begin{equation}
\label{eqn:UV}
\# \cU_{k,\ell}(n) = \binom{k}{\ell} \mand \# \cV_{k,m}(n)= \binom{r-k}{m}.
\end{equation}

Clearly the binary expansion of 
any integer of the shape $u2^{r-k+1} + v$, $u \in \cU_{k,\ell}(n)$, $v \in \cV_{k,m}(n)$
differs from the binary expansion of $n$ in exactly $\ell + m$ positions. 
This suggests to consider the following double sum  
with a  multiplicative character $\chi$ of $\F_p^*$:
$$
S_n(k,\ell,m; \chi) = \sum_{u \in \cU_{k,\ell}(n)}\, \sum_{v \in \cV_{k,m}(n)}\chi(u2^{r-k+1} + v).
$$

We note that the following result is slightly more precise than 
a bound of  Karatsuba~\cite{Kar1} (see also~\cite[Chapter~VIII, Problem~9]{Kar2})
that applies to double character sums over arbitrary sets.

\begin{lem}
\label{lem:Sum general} In the notation from above, for any non-trivial
multiplicative character $\chi$ of $\F_p^*$ and any
positive integer $\nu$, we have
\begin{equation*}
\begin{split}
|S_n(k,\ell,m; \chi)|\ll &\( \# \cU_{k,\ell}(n)\)^{(2\nu-1)/2\nu} \(\# \cV_{k,m}(n)\)^{1/2} 2^{k/2\nu }\\
 &\qquad \quad + 
\(\# \cU_{k,\ell}(n)\)^{(2\nu-1)/2\nu} \# \cV_{k,m}(n) 2^{r/4\nu}(\log p)^{1/2\nu}. 
\end{split}
\end{equation*}
\end{lem}

\begin{proof}
Let $K = 2^k$. 
By the H\" older inequality, we have
\begin{eqnarray*}
\lefteqn{\left|S_n(k,\ell,m; \chi) \right|^{2\nu}} \nonumber\\
&  & \le  \#\cU_{k,\ell}(n)^{2\nu-1}
\sum_{u=0}^{K-1}\left|\sum_{v\in\cV_{k,m}(n)}\chi\(u2^{r-k+1}+v\)\right|^{2\nu}
\\
& &=  \#\cU_{k,\ell}(n)^{2\nu-1}
\sum_{u=0}^{K-1}  \sum_{\substack{v_1,\ldots, v_{\nu}\in\cV_{k,m}(n)\\ w_1,\ldots, w_{\nu}\in\cV_{k,m}(n)}}\prod_{i=1}^{\nu}\chi\(u2^{r-k+1} + v_i\)
\overline\chi\(u2^{r-k+1} + w_i\), \\
\end{eqnarray*}
therefore,
\begin{equation}
\label{eq:Prelim}
\begin{split}
\left|S_n(k,\ell,m; \chi) \right|&^{2\nu} 
  \le \#\cU_{k,\ell}(n)^{2\nu-1} \sum_{\substack{v_1,\ldots, v_{\nu}\in\cV_{k,m}(n)\\ 
w_1,\ldots, w_{\nu}\in\cV_{k,m}(n)}}\\
& \qquad \qquad \qquad \quad \left|\sum_{u=0}^{K-1}\prod_{i=1}^{\nu}
\chi\((u2^{r-k+1} + v_i)(u2^{r-k+1} + w_i)^{p-2}\)\right|.
\end{split}
\end{equation}

We note that if the polynomial
\begin{equation}
\label{eqn:Funct}
\prod_{j=1}^\nu (2^{r-k+1}U+v_j)(2^{r-k+1}U+w_j)^{p-2} \in \F_p[U]
\end{equation}
is a power of another rational function, then every value
that occurs in the sequence $v_1, \ldots, v_\nu$ and in the sequence
$w_1, \ldots, w_\nu$ occurs with multiplicity at least $2$. Thus, the set of such $v_1,\ldots,v_{\nu},w_1,\ldots,w_{\nu}$ takes at most $\nu$ distinct values.

We can assume that $r-k > m$ since otherwise the result is trivial.
So, 
$$
\# \cV_{k,m}(n) = \binom{r-k}{m} \ge 2.
$$
Therefore,   there are at most 
$$ \binom{r-k}{m}+ \binom{r-k}{m}^2 + \cdots + 
\binom{r-k}{m}^\nu \le 2 \binom{r-k}{m}^\nu
$$
subsets of $\cV_{k,m}(n)$ with at most $\nu$ elements. 
When such a subset with $h\le \nu$ elements is fixed, we can obtain  the case
described above by placing its elements into $2\nu$ positions.
This can be done in 
no more than $(2\nu)^h \le (2\nu)^\nu$ ways.
So we have at most $2(2\nu)^\nu \binom{r-k}{m}^\nu$ 
possibilities for vectors $(v_1, \ldots, v_\nu)$ and $(w_1,\ldots,w_{\nu})$ 
such that the polynomial~\eqref{eqn:Funct} is a power of some other rational function. 
Using now Lemma~\ref{lem:Weil-Char} when the rational 
function~\eqref{eqn:Funct} is not a power of another 
rational function, we can estimate
\begin{eqnarray*}
\lefteqn{\sum_{\substack{v_1,\ldots, v_{\nu}\in\cV_{k,m}(n)\\ w_1,\ldots, w_{\nu}\in\cV_{k,m}(n)}}\left|\sum_{u=0}^{K-1}\prod_{i=1}^{\nu}
\chi\((u2^{r-k+1} + v_i)(u2^{r-k+1} + w_i)^{p-2}\)\right|}\\
&&\qquad \qquad\qquad\qquad
\ll \binom{r-k}{m}^{\nu} K+\binom{r-k}{m}^{2\nu}p^{1/2}\log p\\
&&\qquad \qquad\qquad\qquad
=  \#\cV_{k,m}(n)^{\nu} 2^{k} +  \#\cV_{k,m}(n)^{2\nu}p^{1/2}\log p.
\end{eqnarray*}
Recalling~\eqref{eq:Prelim}, we obtain the desired result.
\end{proof}

\section{Proof of Theorem~\ref{thm:Deltap}}
\label{sec4}

We fix some $\rho >\rho_0$ (with $\rho<1/2$)
where $\rho_0$ is the root of the equation~\eqref{eq:rho},  
and some positive
$$
\varepsilon < 1 - \frac{1}{2 H(\rho)}.
$$
We define
$$
k = \fl{(1-\varepsilon)r},\quad  \ell = \fl{\rho k},
\quad m = \fl{0.5(r-k)}.
$$

We now show that  
there exists some $\delta > 0$ such that for any nontrivial
multiplicative character $\chi$ of $\F_p^*$, we have
\begin{equation}
\label{eq:Sklm}
S_n(k,\ell,m; \chi) \ll \# \cU_{k,\ell}(n) \# \cV_{k,m}(n)  p^{-\delta}. 
\end{equation}

Using Lemma~\ref{lem:Sum general} and recalling~\eqref{eqn:UV} we see that 
in order to
establish~\eqref{eq:Sklm}, it is enough to show that for an appropriately chosen $\nu$ 
and some $\eta> 0$ that depends only on $\rho$ and $\varepsilon$, we have
\begin{equation}
\label{eq:Cond V}
\# \cV_{k,m}(n) \ge  2^{(1+\eta)k/2\nu }
\end{equation}
and
\begin{equation}
\label{eq:Cond U}
\# \cU_{k,\ell}(n) \ge 2^{r/2} p^{2\nu\delta} \log p.
\end{equation}

Since by
our choice of parameters,~\eqref{eqn:UV} and Lemma~\ref{lem:Binom} we have 
$$
\# \cV_{k,m}(n)  = 2^{r-k + o(r)} = 2^{\varepsilon r + o(r)},
$$
the bound~\eqref{eq:Cond V} is immediate for all sufficiently 
large $\nu$.

Having fixed $\nu$, we now note that
by~\eqref{eqn:UV} and
Lemma~\ref{lem:Binom} in order to establish~\eqref{eq:Cond U} 
for sufficiently small $\delta>0$,
it is enough to verify that 
$$
H(\rho)(1-\varepsilon)> 1/2,
$$
which holds because of our choice of $\varepsilon$.

We now see that the bound~\eqref{eq:Sklm} holds.
Using  Lemma~\ref{lem:prim root} we now estimate the number $P$ 
of primitive roots modulo $p$
in the set 
$$
\cQ(k,l,m)=\{u2^{r-k+1}+v:u \in \cU_{k,l}(n), v \in \cV_{k,m}(n)\}.
$$
We have
\begin{eqnarray*}
P&=&\sum_{n \in \cQ(k,l,m)}\frac{\varphi(p-1)}{p-1}\sum_{d \mid p-1} 
\frac{\mu(d)}{\varphi(d)}\sum_{\ord \chi=d}\chi(n)\\
&=&\frac{\varphi(p-1)}{p-1}\sum_{d \mid p-1} 
\frac{\mu(d)}{\varphi(d)}\sum_{\ord \chi=d} \sum_{n \in \cQ(k,l,m)}\chi(n)\\
&=&\frac{\varphi(p-1)}{p-1}\# \cQ(k,l,m)+
\frac{\varphi(p-1)}{p-1}\sum_{\substack{d>1:\\ d \mid p-1}} 
\frac{\mu(d)}{\varphi(d)}\sum_{\ord \chi=d} \sum_{n \in \cQ(k,l,m)}\chi(n)\\
&=&\frac{\varphi(p-1)}{p-1}\# \cQ(k,l,m)+O\(
\frac{\varphi(p-1)}{p-1}\sum_{\substack{d>1:\\ d \mid p-1}} 
\frac{1}{\varphi(d)}\sum_{\ord \chi=d} \frac{\# \cQ(k,l,m)}{p^{\delta}}\)\\
&=&\frac{\varphi(p-1)}{p-1}\# \cQ(k,l,m)+
O\left(\sum_{\substack{d>1:\\ d \mid p-1}} 
\frac{1}{\varphi(d)}\sum_{\ord \chi=d} \frac{\# \cQ(k,l,m)}{p^{\delta}}\right)\\
&=&\frac{\varphi(p-1)}{p-1}\# \cQ(k,l,m)+
O\left(\frac{\# \cQ(k,l,m)}{p^{\delta}}\sum_{d \mid p-1} 1\right).
\end{eqnarray*}

Using the well-known bound
$$
\varphi(n) \gg \frac{n}{\log \log (n+2)}    \mand
\sum_{d \mid p-1} 1 = p^{o(1)}
$$
we derive
$$
P \gg  \frac{\# \cQ(k,l,m)}{\log \log (p+1)}.
$$

Hence we conclude that for sufficiently large $p$
the set $\cQ (k,l,m)$ indeed contains primitive roots.
Therefore, for a sufficiently
large $p$ we have 
$$\Delta_p \le m + \ell \le (\rho + \varepsilon/2) r. 
$$
Since $\rho> \rho_0$ and $\varepsilon >0$ are arbitrary 
we obtain the desired result.

\section{Proof of Theorem~\ref{thm:SpareseNonRes}}

We fix some sufficiently small  $\varepsilon>0$ and 
put 
$$
N = \rf{p^{1/(4\sqrt{e})+\varepsilon}}.
$$
By Lemma~\ref{lem:PosPropNonres} there exists $\delta>0$
such that the interval $[1,N]$ contains at least $0.5\delta N$
quadratic non-residues modulo $p$. 
Let 
$$
s = \rf{\frac{\log N}{\log 2}} \mand w = \fl{\(\frac{1}{2} +\varepsilon\) s}.
$$
Note that
\begin{equation}
\label{eq:w and r}
\frac{w}{r} = \(\frac{1}{2} +\varepsilon\)\(\frac{1}{4\sqrt{e}}+\varepsilon\) +o(1),
\end{equation}
as $p\to\infty$. 

Making use of the well-known property that the binary entropy function
has maximum $H(1/2)=1$ and is strictly smaller than $1$ outside $1/2$,
we conclude that
the number of positive integers $n \le N$ with Hamming weight at least $w$,
by  Lemma~\ref{lem:Binom},
does not exceed
$$
\sum_{k=w}^s \binom{s}{k} \le s \binom{s}{w} \le s 2^{sH(w/s)+ o(s)}  
\le s 2^{s H(1/2 +\varepsilon)+ o(s)}\ll 2^{\eta s} \ll N^{\eta}, 
$$
where $\eta < 1$ depends only on $\varepsilon$.
Therefore, there is a quadratic non-residue $n \le N$ of 
 Hamming weight at most $w-1$. Recalling~\eqref{eq:w and r},
 since $\varepsilon$ is arbitrary, we obtain the desired estimate 
 on $w_p$.

\section{Proof of Theorem~\ref{thm:HilbCube}}

We fix some $\varepsilon > 0$ and let $d = \rf{p^{1/5+\varepsilon}}$.
Then by Lemma~\ref{addcomb} with $D=5$, for any
$a_0, a_1, \ldots, a_d \in \F_p$  with 
pairwise distinct $a_1, \ldots, a_d$, 
the set $\cH(a_0; a_1, \ldots, a_d)$ contains  an arithmetic progression
$an + b$, $n =1, \ldots, N$
of length
$$
   N \gg \frac{p^{(1/5+\varepsilon)\times (5/4)}}{(\log p)^{5/4}}
  \gg p^{1/4+\varepsilon}.
$$
In particular $a \ne 0$, so $a$ has an inverse $\overline{a}$ in
$\F_p$.

Thus for any non-principal
multiplicative character $\chi$ of $\F_p^*$, by Lemma~\ref{lem:PVB} 
we have
$$
\sum_{n=1}^N \chi(an + b) = \chi(a) \sum_{n=1}^N \chi(n + \overline{a}b) 
\ll Np^{-\eta},
$$
where $\eta > 0$ depends only on $\varepsilon$. 
Using  Lemma~\ref{lem:prim root}, since $\varepsilon>0$ is arbitrary, we conclude the proof in the same way as
in Section~\ref{sec4}.

\section{Remarks, Experimental Results and Open\newline Problems}

It is certainly natural to expect that $\Delta_p =o(r)=o(\log p)$ 
which follows,
for example, from 
the standard conjectures about gaps between consecutive primitive roots. 
But possibly it is a little easier to prove.

On the other hand, we do not have any nontrivial lower bounds on $\Delta_p$
apart from the following simple observations: 

\begin{enumerate}
\item
By Dirichlet's Theorem, 
asymptotically half of the primes satisfy 
$p \equiv \pm 1 \mod 8$. For these primes $p$,
the residue $2$ is a quadratic residue and therefore no 
primitive root modulo $p$.
Thus also no power of $2$ is a primitive root, therefore there is no $n$
having Hamming distance $1$ to $0$ such that $n$ is a primitive root
modulo $p$. Hence $\Delta_p \ge 2$ for all 
primes $p\equiv \pm 1 \pmod 8$.
\item
On the other hand, a quantitative version of Artin's conjecture
on primitive roots says that the proportion of primes with $2$ as a primitive
root is given by the {\it Artin constant:\/}
\begin{equation}
\label{eq:artin}
A = \prod_{p~\mathrm{prime}} \left(1-\frac{1}{p(p-1)} \right)=0.3739558\ldots.
\end{equation}
This has been confirmed by Hooley~\cite{Hooley}
on the assumption of a certain extension of the Riemann Hypothesis.
Furthermore, Vinogradov~\cite{Vin} has shown unconditionally that the proportion is 
at most $A$, see also a very short proof of this by Wiertelak~\cite{Wier}.
 For a survey on
Artin's original conjecture and its modified version see~\cite{Moree}.

The primes with $2$ as a primitive root 
are precisely those primes with
$W_p=1$. Hence the proportion of these primes is expected to be 0.3739...
Similarly, the odd primes with $(\frac{2}{p})=-1$
are precisely those primes with $w_p$=1, and their asymptotic density
therefore is $1/2$.

\item
Computationally, for most primes $p \leq 10^6$ one has $\Delta_p=2$.
We list the number of primes for $p \leq 10^3, 10^4, 10^5, 10^6$
according to their value of $w_p, W_p$ and $\Delta_p$ (for $w_p$,
only odd $p$ are considered). As usual, $\pi(x)$ 
denotes the number of primes up to $x$.

 \begin{center}
{\it Count of primes}\\
 \vskip 5pt
\scriptsize
\begin{tabular}{|r|r|rrr|rrr|rrr|}
\hline
$j$&$\pi(10^j)$&
$w=1$& $W=1$& $\Delta=1$&
$w=2$& $W=2$& $\Delta=2$&
$w=3$& $W=3$& $\Delta=3$\\ \hline
$3$ & $168$ & $87$ & $68$ & $12$ & $80$ & $100$ & $153$ & $0$ & $0$ & $3$\\ \hline
$4$ & $1229$ &$625$ &$471$ & $75$ & $603$ & $756$ & $1147$ & $0$ & $2$ & $7$\\ \hline
$5$&$9592$&$4808$ &$3604$&$508$ & $4783$ & $5985$ &$9075$ & $0$ & $3$ & $9$\\ \hline
$6$ & $78498$ & $39276$ &$29342$ & $3915$ & $39221$ & $49145$ & $74565$ & $0$ & $11$ & $18$\\ \hline
\end{tabular}
\end{center}

Observe that (for example) $39276+39221+0=\pi(10^6)-1$ 
(we have $-1$  as $p=2$ is 
omitted from consideration).

As an example for the comments above we note that
$$\frac{39276}{78498}\approx
0.500344 \mand \frac{29342}{78498}\approx 0.373792
$$ 
are  very close to $1/2$ and  
the Artin constant $A = 0.3739558\ldots$ given by~\eqref{eq:artin}, respectively.

\item
A computer search for $p \le 3,000,000$ has produced $24$ primes $p$ with
$\Delta_p=3$, but none with $\Delta_p \ge 4$.
We list a table of these $24$ primes  
and all classes $a$ such that
the Hamming distance between $a$ and the closest primitive root 
 is at distance $3$.
 \begin{center}
{\it Primes with $\Delta_p=3$}\\
 \vskip 5pt
\begin{tabular}{|r|l|}
\hline
$p$ & \text{residue classes $a$}\\ \hline
$17$ & $0$, $16$\\ \hline
$67$ & $0$, $1$, $65$\\ \hline
$257$ & $0$, $256$\\ \hline
$1753$ & $0$\\ \hline
$2089$ & $0$\\ \hline
$8209$ & $0$, $8196$\\ \hline
$8233$ & $0$, $8226$\\ \hline
$65537$ & $0$, $65536$\\ \hline
$77351$ & $0$\\ \hline
$111439$ & $0$\\ \hline
$114001$ & $0$\\ \hline
$164449$ & $0$\\ \hline
$239713$ & $0$\\ \hline
$262153$ & $0$, $262144$\\ \hline
$514711$ & $0$\\ \hline
$924841$ & $0$\\ \hline
$929671$ & $0$\\ \hline
$947911$ & $0$\\ \hline
$1316041 $ & $0$\\ \hline
$1894369$ & $0$\\  \hline
$2097169$ & $0$, $2097152$\\ \hline
$2236879$ & $0$\\ \hline
$2493721$ & $0$\\ \hline
$2743711$ & $0$\\ \hline 
\end{tabular}
\end{center}
We note that all the corresponding residue classes are close to either end
of the set of residues. Some of these cases can be explained by observing 
that often $p$ or the corresponding class are close to a power of 2. Then
a small shift may lead to an extra carry of the leading bit.

\item
It is well-known that, assuming the GRH, for some constant $C>0$  and $L = C (\log p)^2$, 
there are at least $0.4 L/\log L$ primes $\ell < L$ that are quadratic non-residues 
modulo $p$, see~\cite[Chapter~13]{Mon}. 
The same counting argument as in the proof of Theorem~\ref{thm:SpareseNonRes}
implies that in this case $w_p \le (1 + o(1)) \log r$. With more work, one 
   can also improve the bound $W_p \le (6 + o(1)) \log r$ that follows under the
GRH from the  Shoup~\cite{Shoup} bound~\eqref{ShoupBound}. This in turn may 
lead to constructions of small sets  $\cG_p$ that are guaranteed to have a 
primitive root modulo $p$.
Note that finding this primitive root requires full factorisation of $p-1$, 
however in several applications (for example, in cryptography or 
combinatorics), one can simply consecutively use all elements of such sets;
see~\cite{Shoup,Shp0} for some related results.

\end{enumerate}

We leave the following questions as {\it open problems\/} for further study.
Numerous similar
questions could be analogously stated.

\begin{ques} Investigate whether one can improve the bound in
Theorem~\ref{thm:Deltap}, assuming the GRH.
\end{ques}

\begin{ques}
What can one say about 
$$
f_i(x)=\frac{1}{\pi (x)}\#\{p\leq x: \Delta_p=i\} ?
$$
Does the limit $\lim_{x \to \infty} f_i(x)$ exist?
\end{ques}

\begin{ques} Examine whether $\Delta_p$ is bounded or not.
\end{ques}

\begin{ques} Examine whether $w_p \leq 2$ for most primes, that is, whether
$$
\lim_{x \to \infty} 
\frac{1}{\pi(x)} \sum_{ p \leq x} w_p = \frac{3}{2}
$$ 
holds.
\end{ques}

There are several unrelated results  
estimating the distance in Hamming metric between some other number theoretic 
objects, such as reduced residues modulo a composite number~\cite{BaSh}
and primes, smooth and other special 
integers~\cite{Bourg,GraShp,HarKat,Shp1,Shp2,Tao}, or other additive properties
of the set of quadratic residues or primitive roots \cite{DartygeandSarkozy,Sarkozy}.
These new and exciting directions definitely deserve more attention.

\section*{Acknowledgement} 

During the preparation of this paper, 
R.~Dietmann was supported by EPSRC Grant EP/I018824/1,
C. Elsholtz by FWF-DK Discrete Mathematics Project W1230-N13,
and I.~Shparlinski by ARC Grant DP1092835. 

Finally, we would like to thank the referee for carefully reading the manuscript.

\end{document}